\documentclass[11pt]{article}
\usepackage{amsmath, pb-diagram}
\usepackage{amssymb}
\usepackage{amscd}
\newenvironment{proof}{\par\noindent{\bf Proof \,}}{$\hfill \Box$\par\bigskip}
\newtheorem{thm}{Theorem}[section]
\newtheorem{lem}[thm]{Lemma}

\newtheorem{cor}[thm]{Corollary}

\newtheorem{df}[thm]{Definition}

\newtheorem{ex}[thm]{Example}

\begin{document}

\title{Factorizations of Matrices Over Projective-free Rings}

\author{H. Chen \thanks{Department of Mathematics, Hangzhou Normal University, Hangzhou,
310036, People's Republic of China, e-mail: huanyinchen@yahoo.cn},
H. Kose\thanks{Department of Mathematics, Ahi Evran University,
Kirsehir, Turkey, handankose@gmail.com}, Y. Kurtulmaz
\thanks{Department of Mathematics, Bilkent University, Ankara, Turkey, yosum@fen.bilkent.edu.tr}}

\maketitle

\begin{abstract} An element of a ring $R$ is called strongly $J^{\#}$-clean
provided that it can be written as the sum of an idempotent and an
element in $J^{\#}(R)$ that commute. We characterize, in this
article, the strongly $J^{\#}$-cleanness of matrices over
projective-free rings. These extend many known results on strongly
clean matrices over commutative local rings.

\vspace{2mm} \noindent {\bf 2010 Mathematics Subject Classification :} 15A13, 15B99, 16L99.\\
\noindent {\bf Key words}: strongly $J^{\#}$-matrix,
characteristic polynomial, projective-free ring.
\end{abstract}

\section{Introduction}

Let $R$ be a ring with an identity. We say that $x\in R$ is
strongly clean provided that there exists an idempotent $e\in R$
such that $x-e\in U(R)$ and $ex=xe$. A ring $R$ is strongly clean
in case every element in $R$ is strongly clean (cf. [9-10]). In
[2, Theorem 12], Borooah, Diesl, and Dorsey provide the following
characterization: Given a commutative local ring $R$ and a monic
polynomial $h\in R[t]$ of degree $n$, the following are
equivalent: $(1)$ $h$ has an $SRC$ factorization in $R[t]$; $(2)$
every $\varphi\in M_n(R)$ which satisfies $h$ is strongly clean.
It is demonstrated in [6, Example 3.1.7] that statement $(1)$ of
the above can not weakened from $SRC$ factorization to $SR$
factorization. The purpose of this paper is to investigate a
subclass of strongly clean rings which behave like such ones but
can be characterized by a kind of $SR$ factorizations, and so get
more explicit factorizations for many class of matrices over
projective-free rings.

Let $J(R)$ be the Jacobson radical of $R$. Set
$$J^{\#}(R)=\{ x\in R~\mid~\exists ~n\in {\Bbb N}~\mbox{such
that}~x^n\in J(R)\}.$$ For instance, let $R=M_2({\Bbb Z}_2)$. Then
$$J^{\#}(R)=\{ \left(
\begin{array}{cc}
0&0\\
0&0
\end{array}
\right), \left(
\begin{array}{cc}
0&1\\
0&0
\end{array}
\right), \left(
\begin{array}{cc}
0&0\\
1&0
\end{array}
\right) \},$$ while $J(R)=0$. Thus, $J^{\#}(R)$ and $J(R)$ are
distinct in general. We say that an element $a\in R$ is strongly
$J^{\#}$-clean provided that there exists an idempotent $e\in R$
such that $a-e\in J^{\#}(R)$ and $ea=ae$. If $R$ is a commutative
ring, then $a\in R$ is strongly $J^{\#}$-clean if and only if
$a\in R$ is strongly $J$-clean (cf. [3]). But they behave
different for matrices over commutative rings. A Jordan-Chevalley
decomposition of $n\times n$ matrix $A$ over an algebraically
closed field (e.g., the field of complex numbers), then $A$ is an
expression of it as a sum: $A=E+W$, where $E$ is semisimple, $W$
is nilpotent, and $E$ and $W$ commute. The Jordan-Chevalley
decomposition is extensively studied in Lie theory and operator
algebra. As a corollary, we will completely determine when an
$n\times n$ matrix over a filed is the sum of an idempotent matrix
and a nilpotent matrix that commute. Thus, the strongly
$J^{\#}$-clean factorizations of matrices over rings is also an
analog of that of Jordan-Chevalley decompositions for matrices
over fields.

We characterize, in this article, the strongly $J^{\#}$-cleanness
of matrices over projective-free rings. Here, a commutative ring
$R$ is projective-free provided that every finitely generated
projective $R$-module is free. For instances, every commutative
local ring, every commutative semi-local ring, every principal
ideal domain, every B$\acute{e}$zout domain (e.g., the ring of all
algebraic integers) and the ring $R[x]$ of all polynomials over a
principal domain $R$ are all projective-free. We will show that
strongly $J^{\#}$-clean matrices over projective-free rings are
completely determined by a kind of $``SC"$-factorizations of the
characteristic polynomials. These extend many known results on
strongly clean matrices to such new factorizations of matrices
over projective-free rings (cf. [1-2] and [5]).

Throughout, all rings with an identity and all modules are unitary
modules. Let $f(t)\in R[t]$. We say that $f(t)$ is a monic
polynomial of degree $n$ if $f(t)=t^n+a_{n-1}t^{n-1}+\cdots
+a_1t+a_0$ where $a_{n-1},\cdots ,a_1,a_0\in R$. We always use
$U(R)$ to denote the set of all units in a ring $R$. If
$\varphi\in M_n(R)$, we use $\chi (\varphi)$ to stand for the
characteristic polynomial $det( tI_n-\varphi)$.

\section{Full Matrices Over Projective-free Rings}

Let $A=\left(
\begin{array}{cc}
1&1\\
1&0
\end{array}
\right)\in M_2({\Bbb Z}_2)$. It is directly verified that $A\in
M_2({\Bbb Z}_2)$ is not strongly $J^{\#}$-clean, though $A$ is
strongly clean. It is hard to determine strongly cleanness even
for matrices over the integers, but completely different situation
is in the strongly $J^{\#}$-clean case. The aim of this section is
to characterize a single strongly $J^{\#}$-clean $n\times n$
matrix over projective-free rings. Let $M$ be a left $R$-module.
We denote the endomorphism ring of $M$ by $end(M)$.

\begin{lem} Let $M$ be a left $R$-module, and let $E=end(M)$, and let $\alpha\in E$. Then the following are equivalent:
\begin{enumerate}
\item [(1)]{\it $\alpha\in E$ is strongly $J^{\#}$-clean.}
\vspace{-.5mm}
\item [(2)]{\it There exists a left $R$-module decomposition $M=P\oplus Q$ where $P$ and $Q$ are $\alpha$-invariant, and
$\alpha|_P\in J^{\#}\big(end(P)\big)$ and $(1_M-\alpha)|_Q\in
J^{\#}\big(end(Q)\big)$.}
\end{enumerate}
\end{lem}
\begin{proof} $(1)\Rightarrow (2)$ Since $\alpha$ is strongly $J^{\#}$-clean in $E$, there exists an idempotent $\pi\in E$ and a $u\in J^{\#}(E)$
such that $\alpha=(1-\pi )+u$ and $\pi u=u\pi$. Thus,
$\pi\alpha=\pi u\in J^{\#}\big(\pi E\pi\big)$. Further, $1-\alpha
=\pi +(-u)$, and so $(1-\pi)(1-\alpha)=(1-\pi )(-u)\in J^{\#}\big(
(1-\pi)E(1-\pi)\big)$. Set $P=M\pi $ and $Q=M(1-\pi )$. Then
$M=P\oplus Q$. As $\alpha\pi=\pi\alpha$, we see that $P$ and $Q$
are $\alpha$-invariant. As $\alpha\pi\in J^{\#}\big(\pi
E\pi\big)$, we can find some $t\in {\Bbb N}$ such that
$(\alpha\pi)^t\in J\big(\pi E\pi\big)$. Let $\gamma\in end(P)$.
For any $x\in M$, it is easy to see that $(x)\pi \big(1_P-\gamma
\big(\alpha|_P\big)^t\big)=(x)\pi
\big(\pi-(\pi\overline{\gamma}\pi) (\pi\alpha\pi)^t\big)$ where
$\overline{\gamma}: M\to M$ given by
$(m)\overline{\gamma}=(m)\pi\gamma$ for any $m\in M$. Hence,
$1_P-\gamma \big(\alpha|_P\big)^t\in aut(P)$. Hence
$\big(\alpha|_P\big)^t\in J\big(end(P)\big)$. This implies that
$\alpha|_P\in J^{\#}\big(end(P)\big)$. Likewise, we verify that
$(1-\alpha)|_Q\in J^{\#}\big(end(Q)\big)$.

$(2)\Rightarrow (1)$ For any $\lambda\in end(Q)$, we construct an
$R$-homomorphism $\overline{\lambda}\in end(M)$ given by
$\big(p+q\big)\overline{\lambda}=(q)\lambda$. By hypothesis,
$\alpha|_P\in J^{\#}\big(end(P)\big)$ and $(1_M-\alpha)|_Q\in
J^{\#}\big(end(Q)\big)$. Thus, $\alpha
=\overline{1_{Q}}+\overline{\alpha
|_{P}}-\overline{(1_M-\alpha)|_{Q}}$. As $P$ and $Q$ are
$\alpha$-invariant, we see that
$\alpha\overline{1_{Q}}=\overline{1_{Q}}\alpha$. In addition,
$\overline{1_{Q}}\in end(M)$ is an idempotent. As
$\big(\overline{\alpha
|_{P}}\big)\big(\overline{(1_M-\alpha)|_{Q}}\big)=0=\big(\overline{(1_M-\alpha)|_{Q}}\big)\big(\overline{\alpha
|_{P}}\big)$, we show that $\overline{\alpha
|_{P}}-\overline{(1_M-\alpha)|_{Q}}\in J^{\#}\big(end(M)\big)$, as
required.
\end{proof}

\begin{lem} Let $R$ be a ring, and let $M$ be a left $R$-module. Suppose that $x,y,a,b\in end(M)$ such that $xa+yb=1_M, xy=yx=0, ay=ya$ and $xb=bx$.
Then $M=ker(x)\oplus ker(y)$ as left $R$-modules.
\end{lem}
\begin{proof} Straightforward. (cf. [6, Lemma 3.2.6]).
\end{proof}
\begin{lem} Let $R$ be a commutative ring, and let $\varphi\in M_n(R)$. Then the following are equivalent:
\begin{enumerate}
\item [(1)]{\it $\varphi\in J^{\#}\big(M_n(R)\big)$.}
\item [(2)]{\it $\chi (\varphi)\equiv t^n \big(mod~ J(R)\big)$, i.e., $\chi (\varphi)-t^n\in J(R)[t]$.}
\item [(3)]{\it There exists a monic polynomial $h\in R[t]$ such that $h\equiv t^{degh} \big(mod~J(R)\big)$ for which $h(\varphi)=0$.}
\end{enumerate}
\end{lem}
\begin{proof} $(1)\Rightarrow (2)$ Since $\varphi\in J^{\#}\big(M_n(R)\big)$, there exists some $m\in {\Bbb N}$ such that $\varphi^m\in
J\big(M_n(R)\big)$. As $J\big(M_n(R)\big)=M_n\big(J(R)\big)$, we
get $\overline{\varphi}\in N\big(M_n(R/J(R))\big)$. In view of [6,
Proposition 3.5.4], $\chi\big(\overline{\varphi}\big)\equiv
t^n\big(mod~N(R/J(R))\big)$. Write
$\chi(\varphi)=t^n+a_1t^{n-1}+\cdots +a_n$. Then
$\chi\big(\overline{\varphi}\big)=t^n+\overline{a_1}t^{n-1}+\cdots
+\overline{a_n}$. We infer that each $a_i^{m_i}+J(R)=0+J(R)$ where
$m_i\in {\Bbb N}$. This implies that $a_i\in J^{\#}(R)$. That is,
$\chi(\varphi)\equiv t^n \big(mod~ J^{\#}(R)\big)$. Obviously,
$J(R)\subseteq J^{\#}(R)$. For any $x\in J^{\#}(R)$, then there
exists some $m\in {\Bbb N}$ such that $x^n\in J(R)$. For any
maximal ideal $M$ of $R$, $M$ is prime, and so $x\in M$. This
implies that $x\in J(R)$; hence, $J^{\#}(R)\subseteq J(R)$.
Therefore $J^{\#}(R)=J(R)$, as required.

$(2)\Rightarrow (3)$ Choose $h=\chi (\varphi)$. Then $h\equiv
t^{degh} \big(mod~J(R)\big)$. In light of the Cayley-Hamilton
Theorem, $h(\varphi)=0$, as required.

$(3)\Rightarrow (1)$ By hypothesis, there exists a monic
polynomial $h\in R[t]$ such that $h\equiv t^{degh}
\big(mod~J(R)\big)$ for which $h(\varphi)=0$. Write
$h=t^n+a_1t^{n-1}+\cdots +a_n$. Choose
$\overline{h}=t^n+\overline{a_1}t^{n-1}+\cdots +\overline{a_n}\in
\big(R/J(R)\big)[t]$. Then $\overline{h}\equiv t^n \big( mod~
N(R/J(R))\big)$ for which
$\overline{h}\big(\overline{\varphi}\big) = 0$. According to [6,
Proposition 3.5.4], there exists some $m\in {\Bbb N}$ such that
$\big(\overline{\varphi}\big)^m=\overline{0}$ over $R/J(R)$.
Therefore $\varphi^m\in M_n\big(J(R)\big)$, and so $\varphi\in
J^{\#}\big(M_n(R)\big)$.
\end{proof}

\begin{df} For $r\in R$, define
$${\Bbb J}_r=\{ f\in R[t]~|~f~\mbox{monic, and}~f\equiv (t-r)^{^{degf}} \big(mod~ J^{{\#}}(R)\big)\}.$$
\end{df}
\begin{lem} Let $R$ be a projective-free ring, let $\varphi\in M_n(R)$, and let $h\in R[t]$ be a monic polynomial of degree $n$. If $h(\varphi)=0$ and
there exists a factorization $h=h_0h_1$ such that $h_0\in {\Bbb
J}_0$ and $h_1\in {\Bbb J}_1$, then $\varphi$ is strongly
$J^{\#}$-clean.
\end{lem}
\begin{proof} Suppose that $h=h_0h_1$ where $h_0\in {\Bbb J}_0$ and $h_1\in {\Bbb J}_1$. Write $h_0=t^p+a_1t^{p-1}+\cdots +a_p$
and $h_1=(t-1)^q+b_1t^{q-1}+\cdots +b_q$. Then each $a_i,b_j\in
J^{\#}(R)$. Since $R$ is commutative, we get each $a_i,b_j\in
J(R)$. Thus, $\overline{h_0}=t^p$ and
$\overline{h_1}=(t-\overline{1})^q$ in $\big(R/J(R)\big)[t]$.
Hence, $\big(\overline{h_0},\overline{h_1}\big)=\overline{1}$, In
virtue of [6, Lemma 3.5.10], we have some $u_0, u_1\in R[t]$ such
that $u_0h_0+u_1h_1=1$. Then
$u_0(\varphi)h_0(\varphi)+u_1(\varphi)h_1(\varphi)=1_{nR}$. By
hypothesis,
$h(\varphi)=h_0(\varphi)h_1(\varphi)=h_1(\varphi)h_0(\varphi)=0$.
Clearly, $u_0(\varphi)h_1(\varphi)=h_1(\varphi)u_0(\varphi)$ and
$h_0(\varphi)u_1(\varphi)=u_1(\varphi)h_0(\varphi)$. In light of
Lemma 2.2, $nR=ker\big(h_0(\varphi)\big)\oplus
ker\big(h_1(\varphi)\big)$. As $h_0t=th_0$ and $h_1t=th_1$, we see
that $h_0(\varphi)\varphi=\varphi h_0(\varphi)$ and
$h_1(\varphi)\varphi=\varphi h_1(\varphi)$, and so
$ker\big(h_0(\varphi)\big)$ and $ker\big(h_1(\varphi)\big)$ are
both $\varphi$-invariant. It is easy to verify that
$h_0\big(\varphi~|_{ker(h_0(\varphi))}\big)=0$. Since $h_0\in
{\Bbb J}_0$, we see that $h_0\equiv t^{degh_0}
\big(mod~J^{\#}(R)\big)$; hence, $\varphi~|_{ker(h_0(\varphi))}\in
J^{\#}\big(end(kerh_0(\varphi))\big)$.

It is easy to verify that
$h_1\big(\varphi~|_{ker(h_1(\varphi))}\big)=0$. Set
$g(u)=(-1)^{degh_1}h_1(1-u)$. Then
$g\big((~1-\varphi~)|_{ker(h_1(\varphi))}\big)=0$. Since $h_1\in
{\Bbb J}_1$, we see that $h_1\equiv
(t-1)^{degh_1}\big(mod~J^{\#}(R)\big)$. Hence, $g(u)\equiv
(-1)^{degh_1}(-u)^{degg}\big(mod~J(R)\big)$. This implies that
$g\in {\Bbb J}_0$. By virtue of Lemma 2.3,
$(1-\varphi)~|_{ker(h_1(\varphi)})\in
J^{\#}\big(end(ker(h_1(\varphi))\big))$. According to Lemma 2.1,
$\varphi\in M_n(R)$ is strongly $J^{\#}$-clean.
\end{proof}

The matrix
$$\varphi =\left(
\begin{array}{ccccc}
0&0&\cdots &0&-a_{0}\\
1&0&\cdots &0&-a_{1}\\
\vdots&\vdots&\ddots &\vdots&\vdots\\
0&0&\cdots &1&-a_{n-1}
\end{array}
\right)\in M_n(R)$$ is called the companion matrix $C_h$ of $h$,
where $h=t^n+a_{n-1}t^{n-1}+\cdots +a_1t+a_0\in R[t]$.

\begin{thm} Let $R$ be a projective-free ring and let $h\in R[t]$ be a monic polynomial of degree $n$. Then the following are equivalent:
\begin{enumerate}
\item [(1)]{\it Every $\varphi\in M_n(R)$ with $\chi (\varphi)=h$ is strongly $J^{\#}$-clean.}
\item [(2)]{\it The companion matrix $C_h$ of $h$ is strongly $J^{\#}$-clean.}
\item [(3)]{\it There exists a factorization $h=h_0h_1$ such that
$h_0\in {\Bbb J}_0$ and $h_1\in {\Bbb J}_1$.}
\end{enumerate}
\end{thm}
\begin{proof} $(1)\Rightarrow (2)$ Write $h=t^n+a_{n-1}t^{n-1}+\cdots +a_1t+a_0\in R[t]$. Choose $$C_h=\left(
\begin{array}{ccccc}
0&0&\cdots &0&-a_{0}\\
1&0&\cdots &0&-a_{1}\\
\vdots&\vdots&\ddots &\vdots&\vdots\\
0&0&\cdots &1&-a_{n-1}
\end{array}
\right)\in M_n(R).$$ Then $\chi (C_h)=h$. By hypothesis, $C_h\in
M_n(R)$ is strongly $J^{\#}$-clean.

$(2)\Rightarrow (3)$ In view of Lemma 2.1, there exists a
decomposition $nR=A\oplus B$ such that $A$ and $B$ are
$\varphi$-invariant, $\varphi~|_{A}\in J^{\#}\big(end_R(A)\big)$
and $(1-\varphi)~|_{B}\in J^{\#}\big(end_R(B)\big)$. Since $R$ is
a projective-free ring, there exist $p,q\in {\Bbb N}$ such that
$A\cong pR$ and $B\cong qR$. Regarding $end_R(A)$ as $M_p(R)$, we
see that $\varphi~|_{A}\in J^{\#}\big(M_p(R)\big)$. By virtue of
Lemma 2.3, $\chi (\varphi~|_{A})\equiv t^p \big(mod~
J^{\#}(R)\big)$. Thus $\chi (\varphi~|_{A})\in {\Bbb J}_0$.
Analogously, $(1-\varphi)~|_{B}\in J^{\#}\big(M_q(R)\big)$. It
follows from Lemma 2.3 that $\chi
\big((1-\varphi)~|_{B}\big)\equiv
t^{q}~\big(~\mbox{mod}~J^{\#}(R)\big)$. This implies that
$det\big(\lambda I_q- (1-\varphi)~|_{B}\big)\equiv
\lambda^{q}~\big(~\mbox{mod}~J^{\#}(R)\big)$. Hence,
 $det\big((1-\lambda)I_q- \varphi~|_{B}\big)\equiv (-\lambda)^{q}~\big(~\mbox{mod}~J^{\#}(R)\big)$. Set $t=1-\lambda$. Then
 $det\big(tI_q- \varphi~|_{B}\big)\equiv (t-1)^{q}~\big(~\mbox{mod}~J^{\#}(R)\big)$. Therefore we get $\chi (\varphi~|_{B})
 \equiv (t-1)^{q}~\big(~\mbox{mod}~J^{\#}(R)\big)$. We infer that $\chi (\varphi~|_{B})\in {\Bbb J}_1$.
 Clearly, $\chi (\varphi)=\chi (\varphi~|_{A})\chi (\varphi~|_{B})$. Choose $h_0=\chi (\varphi~|_{A})$ and $h_1=\chi (\varphi~|_{B})$.
 Then there exists a factorization $h=h_0h_1$ such that $h_0\in {\Bbb J}_0$ and $h_1\in {\Bbb J}_1$, as desired.

$(3)\Rightarrow (1)$ For every $\varphi\in M_n(R)$ with $\chi
(\varphi)=h$, it follows by the Cayley-Hamilton Theorem that
$h(\varphi)=0$. Therefore $\varphi$ is strongly $J^{\#}$-clean by
Lemma 2.5.
\end{proof}

\begin{cor} Let $F$ be a field, and let $A\in M_n(F)$. Then the following are equivalent:
\begin{enumerate}
\item [(1)]{\it $A$ is the sum of an idempotent matrix and a nilpotent matrix that
commute.}
\item [(2)]{\it $\chi(A)=t^s(t-1)^t$ for some $s,t\geq 0$.}
\end{enumerate}
\end{cor}
\begin{proof} As
$J\big(M_n(F)\big)=0$, we see that a $n\times n$ matrix contains
in $J^{\#}\big(M_n(F)\big)$ if and only if $A$ is a nilpotent
matrix. So $A\in M_n(F)$ is strongly $J^{\#}$-clean if and only if
$A$ is the sum of an idempotent matrix and a nilpotent matrix that
commute. By virtue of Theorem 2.6, we see that $A\in M_n(F)$ is
the sum of an idempotent matrix and a nilpotent matrix that
commute if and only if $\chi(A)=h_0h_1$, where $h_0\in {\Bbb J}_0$
and $h_1\in {\Bbb J}_1$. Clearly, $h_0\in {\Bbb J}_0$ if and only
if $h_0\equiv t^{degh_0}(mod~J^{\#}(F))$. But $J^{\#}(F)=0$, and
so $h_0=t^s$, where $s=degh_0$. Likewise, $h_1=(t-1)^t$, where
$t=degh_1$. Therefore we complete the proof.\end{proof}

For matrices over integers ,we have a similar situation. As
$J\big(M_n({\Bbb Z})\big)=0$, we see that an $n\times n$ matrix
contains in $J^{\#}\big(M_n({\Bbb Z})\big)$ if and only if it is a
nilpotent matrix. Likewise, we show that $A\in M_n({\Bbb Z})$ is
the sum of an idempotent matrix and a nilpotent matrix that
commute if and only if $\chi(A)=t^s(t-1)^t$ for some $s,t\geq 0$.
For instance, choose $A=\left(
\begin{array}{ccc}
-2&2&-1\\
-4&4&-2\\
-1&1&0
\end{array}
\right)\in M_3({\Bbb Z})$. Then $\chi(A)=t(t-1)^2$. Thus, $A$ is
the sum of an idempotent matrix and an nilpotent matrix that
commute. In fact, we have a corresponding factorization $A=\left(
\begin{array}{ccc}
-1&1&0\\
-2&2&0\\
0&0&1
\end{array}
\right)+\left(
\begin{array}{ccc}
-1&1&-1\\
-2&2&-2\\
-1&1&-1
\end{array}
\right)$.

\begin{cor} Let $R$ be a projective-free ring, and let $\varphi\in M_2(R)$. Then $\varphi$ is strongly $J^{\#}$-clean if and only if
\begin{enumerate}
\item [(1)]{\it $\chi (\varphi )\equiv t^2 \big(mod~ J(R)\big)$; or}
\vspace{-.5mm}
\item [(2)]{\it $\chi (\varphi )\equiv (t-1)^2 \big(mod~ J(R)\big)$; or}\vspace{-.5mm}
\item [(3)]{\it $\chi (\varphi )$ has a root in $J(R)$ and a root in $1+J(R)$.}\vspace{-.5mm}
\end{enumerate}
\end{cor}
\begin{proof} Suppose that $\varphi$ is strongly $J^{\#}$-clean. By virtue of Theorem 2.6, there exists a factorization $\chi(\varphi )=h_0h_1$ such
that $h_0\in {\Bbb J}_0$ and $h_1\in {\Bbb J}_1$.

Case I. $deg(h_0)=2$ and $deg(h_1)=0$. Then $h_0=\chi(\varphi )=
t^2-tr(\varphi)t+det(\varphi)$ and $h_1=1$. As $h_0\in {\Bbb
J}_0$, it follows from Lemma 2.3 that $\varphi\in
J^{\#}\big(M_2(R)\big)$ or $\chi (\varphi )\equiv t^2 \big(mod~
J(R)\big)$.

Case II. $deg(h_0)=1$ and $deg(h_1)=1$. Then $h_0=t-\alpha$ and
$h_1=t-\beta$. Since $R$ is commutative, $J^{\#}(R)=J(R)$. As
$h_0\in {\Bbb J}_0$, we see that $h_0\equiv t (mod~ J(R))$, and
then $\alpha\in J(R)$. As $h_1\in {\Bbb J}_1$, we see that
$h_1\equiv t-1 (mod~ J(R))$, and then $\beta\in 1+J(R)$. Therefore
$\chi (\varphi )$ has a root in $J(R)$ and a root in $1+J(R)$.

Case III. $deg(h_0)=0$ and $deg(h_1)=2$. Then
$h_1(t)=det\big(tI_2-\varphi\big)\equiv (t-1)^2 (mod~ J(R))$. Set
$u=1-t$. Then $det\big(uI_2-(I_2-\varphi)\big)\equiv u^2 \big(mod~
J(R)\big)$. According to Lemma 2.3, $I_2-\varphi\in
J^{\#}\big(M_2(R)\big)$ or $\chi (\varphi )\equiv (t-1)^2
\big(mod~ J(R)\big)$.

We will suffice to show the converse. If $\chi (\varphi )\equiv
t^2 \big(mod~ J(R)\big)$ or $\chi (\varphi )\equiv (t-1)^2
\big(mod~ J(R)\big)$, then $\varphi\in J^{\#}\big(M_2(R)\big)$ or
$I_2-\varphi \in J^{\#}\big(M_2(R)\big)$. This implies that
$\varphi$ is strongly $J^{\#}$-clean. Otherwise, $\varphi,
I_2-\varphi\not\in J\big(M_2(R)\big)$. In addition, $\chi (\varphi
)$ has a root in $J(R)$ and a root in $1+J(R)$. According to [4,
Theorem 16.4.31], $\varphi$ is strongly $J$-clean, and therefore
it is strongly $J^{\#}$-clean.
\end{proof}

Choose $A=\left(
\begin{array}{cc}
\overline{0}&\overline{2}\\
\overline{1}&\overline{3}
\end{array}
\right)\in M_2\big({\Bbb Z}_4\big)$. It is easy to check that $A,
I_2-A\in M_2\big({\Bbb Z}_4\big)$ are not nilpotent. But
$\chi(A)=t^2+t+2$ has a root $\overline{2}\in J({\Bbb Z}_4)$ and a
root $\overline{1}\in 1+J({\Bbb Z}_4)$. As $J({\Bbb Z}_4)=\{
\overline{0},\overline{2}\}$ is nil, we know that every matrix in
$J^{\#}\big(M_2({\Bbb Z}_4)\big)$ is nilpotent. It follows from
Corollary 2.8 that $A$ is the sum of an idempotent matrix and a
nilpotent matrix that commute. Let ${\Bbb Z}_{(2)}=\{
\frac{m}{n}~|~m,n\in {\Bbb Z}, 2~\nmid~n\}$, and let $A=\left(
\begin{array}{cc}
1&1\\
\frac{2}{9}&0
\end{array}
\right)\in M_2({\Bbb Z}_{(2)})$. Then $J({\Bbb Z}_{(2)})=\{
\frac{2m}{n}~|~m,n\in {\Bbb Z}, 2~\nmid~n\}$. As
$\chi(A)=t^2-t+\frac{2}{9}$ has a root $\frac{1}{3}\in 1+J({\Bbb
Z}_{(2)})$ and a root $\frac{2}{3}\in J({\Bbb Z}_{(2)})$. In light
of Corollary 2.8, $A$ is strongly $J$-clean.

\begin{cor} Let $R$ be a projective-free ring, and let $f(t)=t^2+at+b\in R[t]$ be degree $2$ polynomial with $1+a\in J(R),
b\not\in J(R)$. Then the following are equivalent:
\begin{enumerate}
\item [(1)]{\it Every $\varphi\in M_2(R)$ with $\chi (\varphi )=f(t)$ is strongly $J^{\#}$-clean.}
\vspace{-.5mm}
\item [(2)]{\it There exist $r_1\in J(R)$ and $r_2\in 1+J(R)$ such that $f(r_i)=0$.}
\item [(3)]{\it There exists $r\in J(R)$ such that $f(r)=0$.}
\end{enumerate}
\end{cor}
\begin{proof} $(1)\Rightarrow (2)$ Since every $\varphi\in M_2(R)$ with $\chi (\varphi )=f(t)$ is strongly $J^{\#}$-clean, it follows by Corollary
2.8 that $f(t)=(t-r_1)(t-r_2)$ with $r_1\in J(R), r_2\in 1+J(R)$.

$(2)\Rightarrow (3)$ is trivial.

$(3)\Rightarrow (1)$ As $r^2+ar+b=0$, we see that
$f(t)=(t-r)(t+a+r)$. Clearly, $t-r\in {\Bbb J}_0$. As $1+a+r\in
J(R)$, we see that $t+a+r\in {\Bbb J}_1$. According to Theorem
2.6, we complete the proof.
\end{proof}

Let $\varphi$ be a $3\times 3$ matrix over a commutative ring $R$.
Set $mid(\varphi)=det(I_3-\varphi)-1+tr(\varphi)+det(\varphi).$

\begin{cor} Let $R$ be a projective-free ring, and let $\varphi\in M_3(R)$. Then $\varphi$ is strongly $J^{\#}$-clean if and
only if

{\rm (1)} $\chi (\varphi )\equiv t^3 \big(mod~ J(R)\big)$; or

{\rm (2)} $\chi (\varphi )\equiv (t-1)^3 \big(mod~ J(R)\big)$; or

{\rm (3)} $\chi(\varphi )$ has a root in $1+J(R)$,$tr(\varphi)\in
1+J(R)$,mid$(\varphi)\in J(R), det(\varphi)\in J(R)$;or

{\rm (4)} $\chi (\varphi )$ has a root in $J(R)$, $tr(\varphi)\in
2+J(R)$, mid$(\varphi)\in 1+J(R), det(\varphi)\in J(R)$.

\end{cor}

\begin{proof} Suppose that $\varphi$ is strongly $J^{\#}$-clean. By virtue of Theorem 2.6, there exists a factorization $\chi(\varphi )=h_0h_1$ such that $h_0\in {\Bbb J}_0$ and $h_1\in {\Bbb J}_1$.

Case I. $deg(h_0)=3$ and $deg(h_1)=0$. Then $h_0=\chi(\varphi )$
and $h_1=1$. As $h_0\in {\Bbb J}_0$, it follows from Lemma 2.3
that $\varphi\in J^{\#}\big(M_3(R)\big)$.

Case II. $deg(h_0)=0$ and $deg(h_1)=3$. Then
$h_1(t)=det\big(tI_3-\varphi\big)\equiv (t-1)^3 (mod~ J(R))$. Set
$u=1-t$. Then $det\big(uI_3-(I_3-\varphi)\big)\equiv u^3(mod~
J(R))$. According to Lemma 2.3, $I_3-\varphi\in
J^{\#}\big(M_3(R)\big)$.

Case III. $deg(h_0)=2$ and $deg(h_1)=1$. Then $h_0=t^2+at+b$ and
$h_1=t-\alpha$. As $h_0\in {\Bbb J}_0$, we see that $h_0\equiv t^2
(mod~ J(R))$; hence, $a,b\in J(R)$. As $h_1\in {\Bbb J}_1$, we see
that $h_1\equiv t-1 (mod~ J(R))$; hence, $\alpha\in 1+J(R)$. We
see that $a-\alpha=-tr(\varphi),b-a\alpha=mid(\varphi)$ and
$-b\alpha=-det(\varphi)$. Therefore $tr(\varphi)\in 1+J(R),
mid(\varphi)\in J(R)$ and $det(\varphi)\in J(R)$.

Case IV. $deg(h_0)=1$ and $deg(h_1)=2$. Then $h_0=t-\alpha$ and
$h_1=t^2+at+b$. As $h_0\in {\Bbb J}_0$, we see that $h_0\equiv t
(mod~ J(R))$; hence, $\alpha\in J(R)$. As $h_1\in {\Bbb J}_1$, we
see that $h_1\equiv (t-1)^2 (mod~ J(R))$, and then $a\in -2+J(R)$
and $b\in 1+J(R)$. Obviously,
$\chi(\varphi)=t^3-tr(\varphi)t^2+mid(\varphi)t-det(\varphi)$, and
so $a-\alpha=-tr(\varphi),b-a\alpha=mid(\varphi)$ and
$-b\alpha=-det(\varphi)$. Therefore $tr(\varphi)\in 2+J(R),
mid(\varphi)\in 1+J(R)$ and $det(\varphi)\in J(R)$.

Conversely, if $\chi (\varphi )\equiv t^3 \big(mod~ J(R)\big)$ or
$\chi (\varphi )\equiv (t-1)^3 \big(mod~ J(R)\big)$, then
$\varphi\in J^{\#}\big(M_3(R)\big)$ or $I_3-\varphi \in
J^{\#}\big(M_3(R)\big)$. Hence, $\varphi$ is strongly
$J^{\#}$-clean. Suppose $\chi (\varphi )$ has a root $\alpha\in
1+J(R)$ and $tr(\varphi)\in 1+J(R), det(\varphi)\in J(R)$. Then
$\chi (\varphi)=(t^2+at+b)(t-\alpha)$ for some $a,b\in R$. This
implies that $a-\alpha=-tr(\varphi), -b\alpha=-det(\varphi)$.
Hence, $a,b\in J(R)$. Let $h_0=t^2+at+b$ and $h_1=t-\alpha$. Then
$\chi(\varphi)=h_0h_1$ where $h_0\in {\Bbb J}_0$ and $h_1\in {\Bbb
J}_1$. According to Theorem 2.6, $\varphi$ is strongly
$J^{\#}$-clean.

Suppose $\chi (\varphi )$ has a root $\alpha\in J(R)$ and
$tr(\varphi)\in 2+J(R), mid(\varphi)\in 1+J(R)$ and
$det(\varphi)\in J(R)$. Then $\chi (\varphi)=(t-\alpha)(t^2+at+b)$
for some $a,b\in R$. This implies that $a-\alpha=-tr(\varphi),
b-a\alpha=mid(\varphi)$. Hence, $a\in -2+J(R)$,$b\in 1+J(R)$. Let
$h_0=t-\alpha$ and $h_1=t^2+at+b$. Then $\chi(\varphi)=h_0h_1$
where $h_0\in {\Bbb J}_0$ and $h_1\in {\Bbb J}_1$. According to
Theorem 2.6, $\varphi$ is strongly $J^{\#}$-clean, and we are
done.
\end{proof}

\section{Matrices Over Power Series Rings}

The purpose of this section is to extend the preceding discussion
to matrices over power series rings. We use $R[[x]]$ to stand for
the ring of all power series over $R$. Let
$A(x)=\big(a_{ij}(x)\big)\in M_n\big(R[[x]]\big)$. We use $A(0)$
to stand for $\big(a_{ij}(0)\big)\in M_n(R)$.

\begin{thm} Let $R$ be a projective-free ring, and let $A(x)\in M_2\big(R[[x]])$. Then the following are equivalent:
\begin{enumerate}
\item [(1)]{\it $A(x)\in M_2\big(R[[x]])$ is strongly $J^{\#}$-clean.}
\vspace{-.5mm}
\item [(2)]{\it $A(0)\in M_2(R)$ is strongly $J^{\#}$-clean.}\vspace{-.5mm}
\end{enumerate}
\end{thm}
\begin{proof} $(1)\Rightarrow (2)$ Since $A(x)$ is strongly $J^{\#}$-clean in $M_2\big(R[[x]]\big)$, there exists an $E(x)=E^2(x)\in
M_2\big(R[[x]]\big)$ and a $U(x)\in J^{\#}\big(M_2(R[[x]])\big)$
such that $A(x)=E(x)+U(x)$ and $E(x)U(x)=U(x)E(x)$. This implies
that $A(0)=E(0)+U(0)$ and $E(0)U(0)=U(0)E(0)$ where
$E(0)=E^2(0)\in M_2(R)$ and $U(0)\in J^{\#}\big(M_2(R)\big)$. As a
result, $A(0)$ is strongly $J^{\#}$-clean in $M_2(R)$.

$(2)\Rightarrow (1)$ Construct a ring morphism $\varphi: R[[x]]\to
R, f(x)\mapsto f(0)$. Then $R\cong R[[x]]/kerf$, where $kerf=\{
f(x)~|~f(0)=0\}\subseteq J\big(R[[x]]\big)$. For any finitely
generated projective $R[[x]]$-module $P$,
$P\bigotimes\limits_{R}\big(R[[x]]/kerf\big)$ is a finitely
generated projective $R[[x]]/kerf$-module; hence it is free. Write
$P\bigotimes\limits_{R}\big(R[[x]]/kerf\big)\cong
\big(R[[x]]/kerf\big)^m$ for some $m{\Bbb N}$. Then
$P\bigotimes\limits_{R}\big(R[[x]]/kerf\big)\cong
\big(R[[x]]\big)^m\bigotimes\limits_{R}\big(R[[x]]/kerf\big)$.
That is, $P/P\big(kerf\big)$ $\cong
\big(R[[x]]\big)^m/\big(R[[x]]\big)^m\big(kerf\big)$
wit$kerf\subseteq J\big(R[[x]]\big)$. By Nakayama Theorem, $P\cong
\big(R[[x]]\big)^m$ is free. Thus, $R[[x]]$ is projective-free.
Since $A(0)$ is strongly $J^{\#}$-clean in $M_2(R)$, it follows
from Corollary 2.8 that $A(0)\in J^{\#}\big(M_2(R)\big)$, or
$I_2-A(0)\in J^{\#}\big(M_2(R)\big)$, or the characteristic
polynomial $\chi\big(A(0)\big)=y^2+\mu y+\lambda$ has a root
$\alpha\in 1+J(R)$ and a root $\beta\in J(R)$. If $A(0)\in
J^{\#}\big(M_2(R)\big)$, then $A(x)\in
J^{\#}\big(M_2(R[[x]])\big)$. If $I_2-A(0)\in
J^{\#}\big(M_2(R)\big)$, then $I_2-A(x)\in
J^{\#}\big(M_2(R[[x]])\big)$. Otherwise, we write
$y=\sum\limits_{i=0}^{\infty}b_ix^i$  and $\chi(A(x))=y^2-\mu
(x)y-\lambda (x)$. Then $y^2=\sum\limits_{i=0}^{\infty}c_ix^i$
where $c_i=\sum\limits_{k=0}^{i}b_kb_{i-k}$. Let
$\mu(x)=\sum\limits_{i=0}^{\infty}\mu_ix^i,
\lambda(x)=\sum\limits_{i=0}^{\infty}\lambda_ix^i\in R[[x]]$ where
$\mu_0=\mu$ and $\lambda_0=\lambda$. Then, $y^2-\mu (x)y-\lambda
(x)=0$ holds in $R[[x]]$ if the following equations are satisfied:
$$\begin{array}{c}
b_0^2-b_0\mu_0-\lambda_0=0;\\
(b_0b_1+b_1b_0)-(b_0\mu_1+b_1\mu_0)-\lambda_1=0;\\
(b_0b_2+b_1^2+b_2b_0)-(b_0\mu_2+b_1\mu_1+b_2\mu_0)-\lambda_2=0;\\
\vdots
\end{array}$$ Obviously, $\mu_0=\alpha + \beta\in U(R)$ and $\alpha - \beta\in U(R)$ . Let $b_0=\alpha$. Since
$R$ is commutative, there exists some $b_1\in R$ such that
$$b_0b_1+b_1(b_0-\mu_0)=\lambda_1+b_0\mu_1.$$ Further, there exists some $b_2\in R$ such that
$$b_0b_2+b_2(b_0-\mu_0)=\lambda_2-b_1^2+b_0\mu_2+b_1\mu_1.$$ By
iteration of this process, we get $b_3,b_4,\cdots $. Then $y^2-\mu
(x)y -\lambda (x)=0$ has a root $y_0(x)\in 1+J\big(R[[x]]\big)$.
If $b_0=\beta\in J(R)$, analogously, we show that $y^2-\mu (x)y
-\lambda (x)=0$ has a root $y_1(x)\in J\big(R[[x]]\big)$. In light
of Corollary 2.8, the result follows.
\end{proof}

\begin{cor} Let $R$ be a projective-free ring, and let $A(x)\in M_2\big(R[[x]]/(x^{m})\big)$ $(m\geq 1)$. Then the following
are equivalent:
\begin{enumerate}
\item [(1)]{\it $A(x)\in M_2\big(R[[x]]/(x^m)\big)$ is strongly $J^{\#}$-clean.}
\vspace{-.5mm}
\item [(2)]{\it $A(0)\in M_2(R)$ is strongly $J^{\#}$-clean.}\vspace{-.5mm}
\end{enumerate}
\end{cor}
\begin{proof} $(1)\Rightarrow (2)$ is obvious.

$(2)\Rightarrow (1)$ Let $\psi: R[[x]]\to R[[x]]/(x^m),
\psi(f)=\overline{f}$. Then it reduces a surjective ring
homomorphism $\psi^*: M_2\big(R[[x]]\big)\to
M_2\big(R[[x]]/(x^m)\big)$. Hence, we have a $B\in
M_2\big(R[[x]]\big)$ such that $\psi^*\big(B(x)\big)=A(x)$.
According to Theorem 3.1, we complete the proof.
\end{proof}

\begin{ex} {\rm Let $R={\Bbb Z}_4[x]/(x^2)$, and let $A(x)=\left( \begin{array}{cc}
\overline{2}&\overline{2}+\overline{2}x\\
\overline{2}+x&\overline{3}+\overline{3}x \end{array} \right)$
$\in M_2(R)$. Obviously, ${\Bbb Z}_4$ is a projective-free ring,
and that $R={\Bbb Z}_4[[x]]/(x^2)$. Since we have the strongly
$J^{\#}$-clean decomposition $A(0)=\left(
\begin{array}{cc}
\overline{0}&\overline{2}\\
\overline{2}&\overline{1} \end{array} \right)+\left(
\begin{array}{cc}
\overline{2}&\overline{0}\\
\overline{0}&\overline{2} \end{array} \right)$ in $M_2({\Bbb
Z}_4)$, it follows by Corollary 3.2 that $A(x)\in M_2(R)$ is
strongly $J^{\#}$-clean.}
\end{ex}

\begin{thm} Let $R$ be a projective-free ring, and let $A(x)\in M_3\big(R[[x]])$. Then the following are equivalent:
\begin{enumerate}
\item [(1)]{\it $A(x)\in M_3\big(R[[x]])$ is strongly $J^{\#}$-clean.}
\vspace{-.5mm}
\item [(2)]{\it $A(x)\in M_3\big(R[[x]]/(x^m)\big)(m\geq 1)$ is strongly $J^{\#}$-clean.}
\item [(3)]{\it $A(0)\in M_3(R)$ is strongly $J^{\#}$-clean.}
\vspace{-.5mm}
\end{enumerate}
\end{thm}
\begin{proof} $(1)\Rightarrow (2)$ and $(2)\Rightarrow (3)$ are clear.

$(3)\Rightarrow (1)$ As $A(0)$ is strongly $J^{\#}$-clean in
$M_3(R)$, it follows from Corollary 2.10 that $A(0)\in
J^{\#}\big(M_3(R)\big)$, or $I_3-A(0)\in J^{\#}\big(M_3(R)\big)$,
or $\chi\big(A(0)\big)$ has a root in $J(R)$ and
$tr\big(A(0)\big)\in 2+J(R), mid\big(A(0)\big)\in 1+J(R),
det\big(A(0)\big)\in J(R)$, or $\chi\big(A(0)\big)$ has a root in
$1+J(R)$ and $tr\big(A(0)\big)\in 1+J(R), mid\big(A(0)\big)\in
J(R), det\big(A(0)\big)\in J(R)$. If $A(0)\in
J^{\#}\big(M_3(R)\big)$ or $I_3-A(0)\in J^{\#}\big(M_3(R)\big)$,
then $A(x)\in J^{\#}\big(M_3(R[[x]])\big)$ or $I_3-A(x)\in
J^{\#}\big(M_3(R[[x]])\big)$. Hence, $A(x)\in M_3\big(R[[x]]\big)$
is strongly $J^{\#}$-clean. Assume that
$\chi\big(A(0)\big)=t^3-\mu t^2-\lambda t-\gamma $ has a root
$\alpha\in J(R)$ and $tr\big(A(0)\big)\in
2+J(R),mid\big(A(0)\big)\in 1+J(R),det\big(A(0)\big)\in J(R)$.
Write $y=\sum\limits_{i=0}^{\infty}b_ix^i$. Then
$y^2=\sum\limits_{i=0}^{\infty}c_ix^i$ where
$c_i=\sum\limits_{k=0}^{i}b_kb_{i-k}$. Further,
$y^3=\sum\limits_{i=0}^{\infty}d_ix^i$ where
$d_i=\sum\limits_{k=0}^{i}b_kc_{i-k}$. Let
$\mu(x)=\sum\limits_{i=0}^{\infty}\mu_ix^i,
\lambda(x)=\sum\limits_{i=0}^{\infty}\lambda_ix^i,
\gamma(x)=\sum\limits_{i=0}^{\infty}\gamma_ix^i\in R[[x]]$ where
$\mu_0=\mu,\lambda_0=\lambda$ and $\gamma_0=\gamma$. Then,
$y^3-\mu (x)y^2-\lambda (x)y-\gamma(x)=0$ holds in $R[[x]]$ if the
following equations are satisfied:
$$\begin{array}{c}
b_0^3-b_0^2\mu_0-b_0\lambda_0-\gamma_0=0;\\
(3b_0^2-2b_0\mu_0-\lambda_0)b_1=\gamma_1+b_0^2\mu_1+b_0\lambda_1;\\
(3b_0^2-2b_0\mu_0-\lambda_0)b_2=\gamma_2+b_0^2\mu_2+b_1^2\mu_0+2b_0b_1\mu_1+b_0\lambda_2+b_1\lambda_0-3b_0b_1^2;\\
\vdots
\end{array}$$ Let $b_0=\alpha\in J(R)$. Obviously, $\mu_0=trA(0)\in 2+J(R)$ and $\lambda_0=-midA(0)\in U(R)$.
Hence, $3b_0^2-2b_0\mu_0-\lambda_0\in U(R)$. Thus, we see that
$b_1=(3b_0^2-2b_0\mu_0-\lambda_0)^{-1}(\gamma_1+b_0^2\mu_1+b_0\lambda_1)$
and
$b_2=(3b_0^2-2b_0\mu_0-\lambda_0)^{-1}(\gamma_2+b_0^2\mu_2+b_1^2\mu_0+2b_0b_1\mu_1+b_0\lambda_2+b_1\lambda_0-3b_0b_1^2)$.
By iteration of this process, we get $b_3,b_4,\cdots $. Then
$y^3-\mu (x)y^2 -\lambda (x)y-\gamma(x)=0$ has a root $y_0(x)\in
J\big(R[[x]]\big)$. It follows from $trA(0)\in 2+J(R)$ that
$trA(x)\in 2+J\big(R[[x]]\big)$. Likewise, $midA(x)\in
1+J\big(R[[x]]\big)$. According to Corollary 2.10, $A(x)\in
M_3\big(R[[x]])$ is strongly $J^{\#}$-clean.

Assume that $\chi\big(A(0)\big)$ has a root $1+\alpha\in J(R)$ and
$tr\big(A(0)\big)\in 1+J(R), mid\big(A(0)\big)\in J(R), detA(0)\in
J(R)$. Then $det\big(I_3-A(0)\big)=1-trA(0)+midA(0)-detA(0)\in
J(R)$. Set $B(x)=I_3-A(x)$. Then $\chi\big(B(0)\big)$ has a root
$\alpha\in J(R)$ and $tr\big(B(0)\big)\in 2+J(R), detB(0)\in
J(R)$. This implies that $midB(0)=detA(0)-1+trB(0)+detB(0)\in
1+J(R)$. By the preceding discussion, we see that $B(x)\in
M_3\big(R[[x]])$ is strongly $J^{\#}$-clean, and then we are done.
\end{proof}

From this evidence above, we end this paper by asking the
following question: Let $R$ be a projective-free ring, and let
$A(x)\in M_n\big(R[[x]]) (n\geq 4) $. Do the strongly
$J^{\#}$-cleanness of $A(x)\in M_3\big(R[[x]])$ and $A(0)\in
M_3(R)$ coincide with each other?\\

{\bf\large Acknowledgements}\ \ This research was supported by the
Scientific and Technological Research Council of Turkey (2221
Visiting Scientists Fellowship Programme) and the Natural Science
Foundation of Zhejiang Province (Y6090404).

\end{document}